\documentclass[10pt]{amsart}
\usepackage{amssymb}%,MnSymbol}
\usepackage{amsthm,amsmath}
\usepackage{cite}

\title{Pivotal decompositions of functions}

\author{Jean-Luc Marichal}
\address{Mathematics Research Unit, FSTC, University of Luxembourg \\
6, rue Coudenhove-Kalergi, L-1359 Luxembourg, Luxembourg} \email{jean-luc.marichal[at]uni.lu }

\author{Bruno Teheux}
\address{Mathematics Research Unit, FSTC, University of Luxembourg \\
6, rue Coudenhove-Kalergi, L-1359 Luxembourg, Luxembourg} \email{bruno.teheux[at]uni.lu }

\date{April 28, 2014}

\begin{document}

\theoremstyle{plain}
\newtheorem{theorem}{Theorem}[section]% Supprimer [section] pour une numérotation linéaire
\newtheorem{lemma}[theorem]{Lemma}
\newtheorem{proposition}[theorem]{Proposition}
\newtheorem{corollary}[theorem]{Corollary}
\newtheorem{fact}[theorem]{Fact}
\newtheorem*{main}{Main Theorem}

\theoremstyle{definition}
\newtheorem{definition}[theorem]{Definition}
\newtheorem{example}[theorem]{Example}

\theoremstyle{remark}
\newtheorem*{conjecture}{Conjecture}
\newtheorem{remark}{Remark}
\newtheorem{claim}{Claim}

\newcommand{\N}{\mathbb{N}}                     % pos. integers
\newcommand{\R}{\mathbb{R}}                     % reals
\newcommand{\med}{\mathrm{med}}
\newcommand{\umc}{\mathrm{UMC}}
\newcommand{\bfa}{\mathbf{a}}
\newcommand{\bfb}{\mathbf{b}}
\newcommand{\bfc}{\mathbf{c}}
\newcommand{\bfx}{\mathbf{x}}
\newcommand{\ctop}{\mathbf{1}}
\newcommand{\cbot}{\mathbf{0}}
\newcommand{\id}{\mathrm{id}}
\newcommand{\ess}{\mathrm{ess}{\,}}

\begin{abstract}
We extend the well-known Shannon decomposition of Boolean functions to more general classes of functions. Such decompositions, which we call pivotal decompositions, express the fact that every unary section of a function only depends upon its values at two given elements. Pivotal decompositions appear to hold for various function classes, such as the class of lattice polynomial functions or the class of multilinear polynomial functions. We also define function classes characterized by pivotal decompositions and function classes characterized by their unary members and investigate links between these two concepts.
\end{abstract}

\keywords{Shannon decomposition, pivotal decomposition, Boolean function, pseudo-Boolean function, switching theory, aggregation function.}

\subjclass[2010]{94C10}

\maketitle

%---------------------------------------------------------------------------------------------- Section 1
\section{Introduction}

A remarkable (though immediate) property of Boolean functions is the so-called \emph{Shannon decomposition}, or \emph{Shannon expansion} (see \cite{Sha38}), also called
\emph{pivotal decomposition} \cite{BarPro75}. This property states that, for every Boolean function $f\colon\{0,1\}^n\to\{0,1\}$
and every $k\in [n]=\{1,\ldots,n\}$, the following decomposition formula holds:
\begin{equation}\label{eq:decompB}
f(\bfx) ~=~ x_k\, f(\bfx_k^1)+\overline{x}_k\, f(\bfx_k^0){\,},\qquad\bfx=(x_1,\ldots,x_n)\in\{0,1\}^n,
\end{equation}
where $\overline{x}_k=1-x_k$ and $\bfx_k^a$ is the $n$-tuple whose $i$-th coordinate is $a$, if
$i=k$, and $x_i$, otherwise. Here the `$+$' sign represents the classical addition for real numbers.

Decomposition formula (\ref{eq:decompB}) means that we can precompute the function values for $x_k=0$ and $x_k=1$ and then select the
appropriate value depending on the value of $x_k$. By analogy with the cofactor expansion formula for determinants, here $f(\bfx_k^1)$ (resp.\
$f(\bfx_k^0)$) is called the cofactor of $x_k$ (resp.\ $\overline{x}_k$) for $f$ and is derived by setting $x_k=1$ (resp.\ $x_k=0$) in $f$.

Clearly, the addition operation in (\ref{eq:decompB}) can be replaced with the maximum operation $\vee$, thus yielding the following alternative
formulation of (\ref{eq:decompB}):
$$
f(\bfx) ~=~ (x_k\, f(\bfx_k^1))\vee (\overline{x}_k\, f(\bfx_k^0)){\,},\qquad\bfx\in\{0,1\}^n,~k\in [n].
$$
Equivalently, (\ref{eq:decompB}) can also be put in the form
\begin{equation}\label{eq:decompB2}
f(\bfx) ~=~ (x_k\vee f(\bfx_k^0)){\,}(\overline{x}_k\vee f(\bfx_k^1)){\,},\qquad\bfx\in\{0,1\}^n,~k\in [n].
\end{equation}

As it is well known, repeated applications of (\ref{eq:decompB}) show that any $n$-ary Boolean function can always be expressed as the
multilinear polynomial function
\begin{equation}\label{eq:sum-prodBf}
f(\bfx) ~=~ \sum_{S\subseteq [n]} f(\mathbf{1}_S)\,\prod_{i\in S}x_i\,\prod_{i\in [n]\setminus S}\overline{x}_i{\,},\qquad\bfx\in\{0,1\}^n,
\end{equation}
where $\mathbf{1}_S$ is the characteristic vector of $S$ in $\{0,1\}^n$, that is, the $n$-tuple whose $i$-th coordinate is $1$, if $i\in S$, and
$0$, otherwise.

If $f$ is nondecreasing (i.e., the map $z\mapsto f(\bfx_k^z)$ is isotone for every $\bfx\in\{0,1\}^n$ and every $k\in [n]$), then by expanding (\ref{eq:decompB2}) we see that the decomposition formula reduces to
\begin{equation}\label{eq:decompBnd}
f(\bfx) ~=~ (x_k\, f(\bfx_k^1))\vee f(\bfx_k^0){\,},\qquad\bfx\in\{0,1\}^n,~k\in [n],
\end{equation}
or, equivalently,
\begin{equation}\label{eq:decompBnd2}
f(\bfx) ~=~ \med(x_k,f(\bfx_k^1),f(\bfx_k^0)){\,},\qquad\bfx\in\{0,1\}^n,~k\in [n],
\end{equation}
where $\med$ is the ternary median operation defined by
$$
\med(x,y,z) ~=~ (x\wedge y)\vee(y\wedge z)\vee(z\wedge x)
$$
and $\wedge$ is the minimum operation.

Interestingly, the following decomposition formula also holds for nondecreasing $n$-ary Boolean functions:
\begin{equation}\label{eq:decompBnd3}
f(\bfx) ~=~ x_k\,(f(\bfx_k^1)\vee f(\bfx_k^0))+\overline{x}_k\,(f(\bfx_k^1)\wedge f(\bfx_k^0)){\,},\qquad\bfx\in\{0,1\}^n,~k\in [n].
\end{equation}
%or, equivalently,
%$$
%f(\bfx) ~=~ (f(\bfx_k^0)\wedge f(\bfx_k^1))+x_k{\,}|f(\bfx_k^1)-f(\bfx_k^0)|{\,},\qquad\bfx\in\{0,1\}^n.
%$$

Actually, any of the decomposition formulas (\ref{eq:decompBnd})--(\ref{eq:decompBnd3}) exactly expresses the fact that $f$ should be
nondecreasing and hence characterizes the subclass of nondecreasing $n$-ary Boolean functions. We state this result as follows.

\begin{proposition}
A Boolean function $f\colon\{0,1\}^n\to\{0,1\}$ is nondecreasing if and only if it satisfies any of the decomposition formulas
(\ref{eq:decompBnd})--(\ref{eq:decompBnd3}).
\end{proposition}

Decomposition property (\ref{eq:decompB}) also holds for functions $f\colon\{0,1\}^n\to\R$, called \emph{$n$-ary pseudo-Boolean functions}. As a
consequence, these functions also have the representation given in (\ref{eq:sum-prodBf}). Moreover, formula (\ref{eq:decompBnd3}) clearly
characterizes the subclass of nondecreasing $n$-ary pseudo-Boolean functions.

The \emph{multilinear extension} of a pseudo-Boolean function $f\colon\{0,1\}^n\to\R$ is the function $\hat f\colon [0,1]^n\to\R$
defined by (see Owen~\cite{Owe72,Owe88})
\begin{equation}\label{eq:MLE}
\hat f(\bfx) ~=~ \sum_{S\subseteq [n]} f(\mathbf{1}_S)\,\prod_{i\in S}x_i\,\prod_{i\in [n]\setminus S}(1-x_i){\,},\qquad\bfx\in [0,1]^n.
\end{equation}
Actually, a function is the multilinear extension of a pseudo-Boolean function if and only if it is a multilinear polynomial function, i.e., a polynomial function of degree $\leqslant 1$ in each variable. Thus defined, one can easily see that the class of multilinear polynomial functions can be characterized as follows.

\begin{proposition}\label{prop:MLE}
A function $f\colon [0,1]^n\to\R$ is a multilinear polynomial function if and only if it satisfies
\begin{equation}\label{eq:MLE2}
f(\bfx) ~=~ x_k\, f(\bfx_k^1)+(1-x_k)\, f(\bfx_k^0){\,},\qquad \bfx\in [0,1]^n,~k\in [n].
\end{equation}
\end{proposition}

Interestingly, Eq.~(\ref{eq:MLE2}) provides an immediate proof of the property
$$
\frac{\partial f(\bfx)}{\partial{x_k}} ~=~ f(\bfx_k^1)-f(\bfx_k^0),
$$
which holds for every multilinear polynomial function $f\colon [0,1]^n\to\R$.

As far as nondecreasing multilinear polynomial functions are concerned, we have the following characterization, which is a special case of Corollary~\ref{cor:q5w76e5}. Recall first that a multilinear polynomial function is nondecreasing if and only if so is its restriction to $\{0,1\}^n$ (i.e., its defining pseudo-Boolean function).

\begin{proposition}\label{prop:ND-MLE}
A function $f\colon [0,1]^n\to\R$ is a nondecreasing multilinear polynomial function if and only if it satisfies
\begin{equation}\label{eq:MLEnd}
f(\bfx) ~=~ x_k\,(f(\bfx_k^1)\vee f(\bfx_k^0))+\overline{x}_k\,(f(\bfx_k^1)\wedge f(\bfx_k^0)){\,},\qquad \bfx\in [0,1]^n,~k\in [n].
\end{equation}
\end{proposition}

The decomposition formulas considered in this introduction share an interesting common feature, namely the fact that any variable, here denoted $x_k$ and called \emph{pivot}, can be pulled out of the function, reducing the evaluation of $f(\bfx)$ to the evaluation of a function of $x_k$, $f(\bfx_k^1)$, and $f(\bfx_k^0)$.\footnote{In applications, such as cooperative game theory or aggregation function theory, this means that, in a sense, one can isolate the marginal contribution of a factor (attribute, criterion, etc.) from the others.} This feature may be useful when for instance the values $f(\bfx_k^1)$ and $f(\bfx_k^0)$ are much easier to compute than that of $f(\bfx)$. In addition to this, such (pivotal) decompositions may facilitate inductive proofs and may lead to canonical forms such as (\ref{eq:sum-prodBf}).

In this paper we define a general concept of pivotal decomposition for various functions $f\colon X^n\to Y$, where $X$ and $Y$ are nonempty sets (Section 2). We also introduce function classes that are characterized by pivotal decompositions (Section 3) and function classes that are characterized by their unary members and investigate relationships between these concepts (Section 4). We also introduce a natural generalization of the concept of pivotal decomposition, namely componentwise pivotal decomposition (Section 5). We then end our paper by some concluding remarks (Section 6).

%---------------------------------------------------------------------------------------------- Section 2
\section{Pivotal decompositions}

The examples presented in the introduction motivate the following definition. Let $X$ and $Y$ be nonempty sets and let $0$ and $1$ be two fixed elements of $X$. For every function $f\colon X^n\to Y$, define $R_f=\{(f(\bfx_k^1),f(\bfx_k^0)):\bfx\in X^n,~k\in [n]\}$. Throughout we assume that $n\geqslant 1$.

\begin{definition}\label{de:345}
We say that a function $f\colon X^n\to Y$ is \emph{pivotally decomposable} if there exist a subset $D$ of $X\times Y^2$ and a function $\Pi\colon D\to Y$, called \emph{pivotal function}, such that $D\supseteq X\times R_f$ and
\begin{equation}\label{eq:pivdecomp}
f(\bfx) ~=~ \Pi(x_k,f(\bfx_k^1),f(\bfx_k^0)){\,},\qquad \bfx\in X^n,~k\in [n].
\end{equation}
In this case, we say that $f$ is \emph{$\Pi$-decomposable}.
\end{definition}

From Definition~\ref{de:345} we immediately obtain the following two results.

\begin{fact}\label{fact:FcEq}
Let $f\colon X^n\to Y$ be a $\Pi$-decomposable function for some pivotal function $\Pi$. Then, for every $(u,v)\in R_f$, we have $\Pi(1,u,v)=u$ and $\Pi(0,u,v)=v$.
\end{fact}

\begin{proposition}[Uniqueness of the pivotal function]\label{prop:Uniq}
If $f\colon X^n\to Y$ is $\Pi$- and $\Pi'$-decomposable for some pivotal functions $\Pi$ and $\Pi'$, then $\Pi$ and $\Pi'$ coincide on $X\times R_f$.
\end{proposition}

\begin{proof}
Let $(p,u,v)\in X\times R_f$. By definition of $R_f$, there exist $\bfx\in X^n$ and $k\in [n]$ such that $(u,v)=(f(\bfx_k^1),f(\bfx_k^0))$. We then have
$$
\Pi'(p,u,v) ~=~ \Pi'(p,f(\bfx_k^1),f(\bfx_k^0)) ~=~ f(\bfx_k^p) ~=~ \Pi(p,f(\bfx_k^1),f(\bfx_k^0)) ~=~ \Pi(p,u,v),
$$
which completes the proof.
\end{proof}

\begin{example}\label{ex:BF}
Every Boolean function is $\Pi$-decomposable, where $\Pi\colon \{0,1\}^3\to \{0,1\}$ is the classical `if-then-else' connective defined by $\Pi(p,u,v)=(p\wedge u)\vee(\overline{p}\wedge v)$. If $f$ is nondecreasing, we can restrict $\Pi$ to $\{0,1\}\times\{(u,v)\in \{0,1\}^2:u\geqslant v\}$ or consider $\Pi'(p,u,v)=(p\wedge (u\vee v))\vee(\overline{p}\wedge (u\wedge v))$ on $\{0,1\}^3$.
\end{example}

\begin{example}\label{ex:MLE}
Every multilinear polynomial function $f\colon [0,1]^n\to\R$ is $\Pi$-decompo{\-}sable, where $\Pi\colon D\to\R$ is defined by $D=[0,1]\times\R^2$ and $\Pi(p,u,v)=p\, u+(1-p)\, v$. If $f$ is nondecreasing, we can restrict $\Pi$ to $[0,1]\times\{(u,v)\in \R^2:u\geqslant v\}$ or consider $\Pi'(p,u,v)=p\, (u\vee v)+(1-p)(u\wedge v)$ on $[0,1]^n$.
\end{example}

\begin{example}\label{ex:LP}
Let $X$ be a bounded distributive lattice, with $0$ and $1$ as bottom and top elements, respectively. A \emph{lattice polynomial function} on $X$ is a composition of projections, constant functions, and the fundamental lattice
operations $\wedge$ and $\vee$; see, e.g., \cite{CouMar1,CouMar2,CouMar0}. An $n$-ary function $f\colon X^n\to X$ is a lattice polynomial function if and only if it can be written
in the (disjunctive normal) form
$$
f(\bfx) ~=~ \bigvee_{S\subseteq [n]}f(\mathbf{1}_S)\wedge\bigwedge_{i\in S}x_i{\,},\qquad\bfx\in X^n.
$$
It is known \cite{CouMar0,Cou11,Mar09} that a function $f\colon X^n\to X$ is a lattice polynomial function if and only if it is $\Pi$-decomposable, where $\Pi\colon X^3\to X$ is defined by $\Pi(p,u,v)=\med(p,u,v)$.
\end{example}

\begin{example}\label{ex:QLP}
Let $X$ and $Y$ be bounded distributive lattices. We denote by $0$ and $1$ their bottom and top elements, respectively. A function $f\colon X^n\to Y$ is of the form $f=g\circ (\phi,\ldots,\phi)$, where $g\colon Y^n\to Y$ is a lattice polynomial function and $\phi\colon X\to Y$ is a unary function such that $\phi(x)=\med(\phi(x),\phi(1),\phi(0))$, if and only if it is $\Pi$-decomposable, where $\Pi\colon X\times Y^2\to Y$ is defined by $\Pi(p,u,v)=\med(f(p,\ldots,p),u,v)$; see \cite{CouMar10}.
\end{example}

\begin{example}\label{ex:tnorm}
A \emph{t-norm} is a binary function $T\colon [0,1]^2\to [0,1]$ that is symmetric, nondecreasing, associative, and such that $T(1,x)=x$ (see, e.g., \cite{SchSkl05}). Every t-norm $T\colon [0,1]^2\to [0,1]$ is $\Pi$-decomposable with $\Pi\colon [0,1]^3\to\R$ defined by $\Pi(p,u,v)=T(p,u)$.
\end{example}

\begin{example}
Consider a function $f\colon X^n\to Y$, a pivotal function $\Pi\colon X\times Y^2\to Y$, and one-to-one functions $\phi\colon X\to X$ and $\psi\colon Y\to Y$ such that $\phi(0)=0$ and $\phi(1)=1$. One can easily show that $f$ is $\Pi$-decomposable if and only if the function $f'=\psi\circ f\circ (\phi,\ldots,\phi)$ is $\Pi'$-decomposable, where $\Pi'=\psi\circ\Pi\circ (\phi,\psi^{-1},\psi^{-1})$. In particular, if $Y=X$ and $\psi=\phi^{-1}$, we obtain $\Pi'=\phi^{-1}\circ\Pi\circ (\phi,\phi,\phi)$. For instance, \emph{quasi-linear functions} $f\colon \R^n\to\R$, defined by (see, e.g., \cite{Acz66})
$$
f(\bfx) ~=~ \phi^{-1}\bigg(\sum_{i=1}^n a_i\,\phi(x_i) + b\bigg),
$$
where $a_1,\ldots,a_n,b\in\R$, are pivotally decomposable.
\end{example}

Repeated applications of (\ref{eq:pivdecomp}) lead to the following fact.

\begin{fact}\label{fact:Restr01}
Let $\Pi$ be a pivotal function. Any $\Pi$-decomposable function $f\colon X^n\to Y$ is uniquely determined by $\Pi$ and the restriction of $f$ to $\{0,1\}^n$.
\end{fact}

A \emph{section} of $f\colon X^n\to Y$ is a function which can be obtained from $f$ by replacing some of its variables by constants. Formally, for every $S\subseteq [n]$ and every $\bfa\in X^n$, we define the $S$-section $f_S^{\bfa}\colon X^S\to Y$ of $f$ by $f_S^{\bfa}(\bfx)=f(\bfa_S^{\bfx})$, where $\bfa_S^{\bfx}$ is the $n$-tuple whose $i$-th coordinate is $x_i$, if
$i\in S$, and $a_i$, otherwise. We also denote $f_{\{k\}}^{\bfa}$ by $f_k^{\bfa}$.

\begin{fact}\label{rem:bij}
Eq.~(\ref{eq:pivdecomp}) implies that, for every fixed $\bfa,\bfb\in X^n$ and $k\in [n]$, we have $f_k^{\bfa}=f_k^{\bfb}$ if and only if $(f(\bfa_k^1),f(\bfa_k^0))=(f(\bfb_k^1),f(\bfb_k^0))$.
\end{fact}

\begin{fact}\label{fact:45j4h}
If a function $f\colon X^n\to Y$ is $\Pi$-decomposable for some pivotal function $\Pi$, then every section of $f$ is also $\Pi$-decomposable.
\end{fact}

%\begin{proposition}
%Let $C$ be a class of functions $f\colon [0,1]^n\to\R$ $(n\geqslant 1)$ such that
%\begin{enumerate}
%\item[(i)] the unary members of $C$ are $q$-decomposable;
%
%\item[(ii)] for $n>1$, any unary section of an $n$-ary function $f$ in $C$ is also in $C$.
%\end{enumerate}
%Then every function in $C$ is $q$-decomposable.
%\end{proposition}
%
%\begin{proof}
%Let $C$ be a class of functions satisfying the conditions of the proposition. We show that each $f\colon [0,1]^n\to\R$ in $C$ is $q$-decomposable. By condition (i), the claim holds for $n=1$. So suppose that $n>1$ and let $k\in [n]$ and $\bfa\in [0,1]^n$. By condition (ii), we have that $f^{\bfa}_k\in C$, and hence $f^{\bfa}_k$ is $q$-decomposable  by condition (i). Therefore, we have
%$$
%f(\bfa) ~=~ f^{\bfa}_k(a_k) ~=~ q(f^{\bfa}_k(0),a_k,f^{\bfa}_k(1)) ~=~ q(f(\bfa_k^0),a_k,f(\bfa_k^1)).
%$$
%Since the above holds for every $\bfa\in [0,1]^n$, it follows that $f$ is $q$-decomposable.
%\end{proof}

\begin{proposition}\label{prop:fd7ssdf}
A function $f\colon X^n\to Y$ is $\Pi$-decomposable for some pivotal function $\Pi$ if and only if so are its unary sections.
\end{proposition}

\begin{proof}
(Necessity) Follows from Fact~\ref{fact:45j4h}.

(Sufficiency) Let $f\colon X^n\to Y$ be $\Pi$-decomposable. For every $\bfx\in X^n$ and every $k\in [n]$, we then have
$$
f(\bfx) ~=~ f(\bfx_k^{x_k}) ~=~ f_k^{\bfx}(x_k) ~=~ \Pi(x_k,f_k^{\bfx}(1),f_k^{\bfx}(0)) ~=~ \Pi(x_k,f(\bfx_k^1),f(\bfx_k^0)),
$$
which completes the proof.
\end{proof}

%---------------------------------------------------------------------------------------------- Section 3
\section{Pivotally characterized classes of functions}

The examples given in the previous sections motivate the consideration of function classes that are characterized by given pivotal functions. The fact that any section of a pivotally decomposable function is also pivotally decomposable with the same pivotal function suggests considering classes of functions with unbounded arities.

The $k$-th argument of a function $f\colon X^n\to Y$ is said to be \emph{inessential} if $f_k^{\bfa}$ is constant for every $\bfa\in X^n$ (see \cite{Sal63}). Otherwise, it is said to be \emph{essential}. We say that a unary section $f_k^{\bfa}$ of $f$ is \emph{essential} if the $k$-th argument of $f$ is essential.

It is natural to ask that a function class characterized by a pivotal function be closed under permuting arguments or adding, deleting, or identifying inessential arguments of functions. We then consider the following definition.

For every function $f\colon X^n \to Y$ and every map $\sigma\colon [n] \rightarrow [m]$, define the function $f_\sigma\colon X^m \rightarrow Y$ by $f_{\sigma}(\bfa)=f(\bfa\sigma)$, where $\bfa\sigma$ denotes the $n$-tuple $(a_{\sigma(1)},\ldots,a_{\sigma(n)})$. Define also the set $U = \bigcup_{n\geqslant 1}Y^{X^n}$.

\begin{definition}
Define an equivalence relation on $U$ as follows. For functions $f\colon X^n\to Y$ and $g\colon X^m\to Y$, we say that $f$ and $g$ are \emph{equivalent} and we write $f\equiv g$ if $f$ can be obtained from $g$ by permuting arguments or by adding, deleting, or identifying inessential arguments. Formally, we have $f\equiv g$ if there exist maps $\sigma\colon [m]\to [n]$ and $\mu\colon [n]\to [m]$ such that $f=g_{\sigma}$ and $g=f_{\mu}$.
\end{definition}

Note that if $f\equiv g$, then $f$ and $g$ have the same number of essential arguments. Also, a nonconstant function is always equivalent to a function with no inessential argument.

An $S$-section of $f\colon X^n\to Y$ is said to be \emph{essential} if there exists $\bfb\in X^n$ such that $f_S^{\bfb}$ is nonconstant.

\begin{lemma}\label{lemma:d4fds}
Let $f, g \in U$.
\begin{enumerate}
\item[(i)] If $f\equiv g$, then any section of $f$ is equivalent to a section of $g$.
\item[(ii)] If every section of $f$ is equivalent to a section of $g$ and if every section of $g$ is equivalent to a section of $f$, then $f\equiv g$.
\item[(iii)] If $f$ and $g$ are nonconstant functions, if every essential section of $f$ is equivalent to a section of $g$, and if every essential section of $g$ is equivalent to a section of $f$, then $f \equiv g$.
\end{enumerate}
\end{lemma}
\begin{proof}
Assume that $f\colon X^n \to Y$ and $g\colon X^m \to Y.$

(i) Let $\sigma\colon [m]\rightarrow [n]$ and $\mu\colon [n]\rightarrow [m]$ such that $f=g_\sigma$ and $g=f_\mu$. If $S$ is a nonempty subset of $[n]$ and if $\bfa \in X^n$, it is easy to prove that $f_{S}^{\bfa}=(g_{\sigma^{-1}(S)}^{\bfa\sigma})_{\sigma}$ and that $g_{\sigma^{-1}(S)}^{\bfa\sigma}=(f_S^{\bfa})_\mu$.

(ii) For $h \in U$, let us denote by $\ess h$ the number of essential arguments of $h$. Let $S\subseteq [n]$, $T\subseteq [m]$, $\bfa \in X^n$, $\bfb \in X^m$ such that $g\equiv f_S^{\bfa}$ and $f\equiv g_T^{\bfb}$. It follows that
\[
\ess f=\ess  g_T^{\bfb} \leqslant \ess g = \ess  f_S^{\bfa} \leqslant \ess f.
\]
We conclude that $\ess g=\ess g_T^{\bfb}$ and so that $g\equiv g_T^{\bfb}\equiv f$.

(iii) Since $f$ and $g$ are nonconstant functions, they are their own essential sections and we can complete the proof as in (ii).
\end{proof}

\begin{definition}
Let $\Pi\colon D\to Y$ be a pivotal function. We denote by $\Gamma_{\Pi}$ the subclass of $U$ of functions which are equivalent to $\Pi$-decomposable functions with no essential argument or no inessential argument. We say that a class $C\subseteq U$ is \emph{pivotally characterized} if there exists a pivotal function $\Pi$ such that $C=\Gamma_\Pi$. In that case, we say that $C$ is \emph{$\Pi$-characterized}.
\end{definition}

\begin{proposition}\label{prop:pivotal_unary}
Let $\Pi$ be a pivotal function.
\begin{enumerate}
\item[(i)] A nonconstant function is in $\Gamma_{\Pi}$ if and only if so are its essential unary sections.
\item[(ii)] A constant function $c$ is in $\Gamma_{\Pi}$ if and only if $\Pi(p,c,c)=c$ for every $p\in X$.
\end{enumerate}
\end{proposition}

\begin{proof}
Assertion (ii) is trivial. Let us prove assertion (i).

(Necessity) Suppose that the nonconstant function $f\colon X^n\to Y$ is in $\Gamma_{\Pi}$. Then $f$ is equivalent to a $\Pi$-decomposable function $g\colon X^m\to Y$ with no inessential argument. Let $\bfa\in X^n$ and $k\in [n]$ such that $f^{\bfa}_k$ is an essential unary section of $f$. By Lemma~\ref{lemma:d4fds} (ii), $f^{\bfa}_k$ is equivalent
%\footnote{Formally, assume that $f=g_{\sigma}$ and $g=f_{\mu}$ for some $\sigma\colon [m]\to [n]$ and $\mu\colon [n]\to [m]$. There is only one element $k'\in\sigma^{-1}(k)$ that corresponds to an essential argument of $g$. It follows that $f_k^{\bfa}=\big(g_{\sigma^{-1}(k)}^{\bfa\sigma}\big)_{\sigma'}$ and $g_{\sigma^{-1}(k)}^{\bfa\sigma}=\big(f_k^{\bfa}\big)_{\mu'}$ with $\sigma'\colon\sigma^{-1}(k)\to [1]:j\mapsto 1$ and $\mu'\colon [1]\to\sigma^{-1}(k):1\mapsto k'$.}
to a section $h$ of $g$, which is $\Pi$-decomposable by Fact~\ref{fact:45j4h}. In turn, function $h$ is equivalent to a $\Pi$-decomposable function with no essential or inessential argument.

(Sufficiency) Let us suppose that every essential unary section of a nonconstant function $f\colon X^n\to Y$ is equivalent to a $\Pi$-decomposable function and let us prove that $f$ is also equivalent to a $\Pi$-decomposable function.

Let $\bfa\in X^n$ and let $k\in [n]$ be such that the $k$-th argument of $f$ is essential. Then the essential unary section $f^{\bfa}_k$ is equivalent to a $\Pi$-decomposable function $g\colon X^m\to Y$ with at most one essential argument. Hence, there is a map $\mu\colon [1]\to [m]$ such that $g(\mathbf{c})=f^{\bfa}_k(\mathbf{c}\mu)$ for every $\mathbf{c}\in X^m$. For any fixed $\mathbf{c}\in X^m$ and every $p\in X$, we then have
$$
f^{\bfa}_k(p) ~=~ g(\bfc_{\mu(1)}^p) ~=~ \Pi(p,g(\bfc_{\mu(1)}^1),g(\bfc_{\mu(1)}^0)) ~=~ \Pi(p,f^{\bfa}_k(1),f^{\bfa}_k(0)),
$$
which shows that every essential unary section of $f$ is $\Pi$-decomposable.

Now, let $E\subseteq [n]$ be the nonempty set of labels of essential arguments of $f$ and let $h\colon X^E\to Y$ be the function obtained from $f$ by deleting its inessential arguments.\footnote{The function $h$ can be defined formally as follows. Let $\sigma\colon [n]\to E$ be any extension to $[n]$ of the map $\iota\colon E\to [n]:k\mapsto k$. Then $h$ is defined by $h(\bfa)=f(\bfa\sigma)$ for every $\bfa\in X^E$. Since $f(\bfb)=h(\bfb\iota)$ for every $\bfb\in X^n$, the functions $f$ and $h$ are equivalent.} Thus, $h$ is equivalent to $f$ and has no inessential arguments. Moreover, its unary sections are essential unary sections of $f$ and hence are $\Pi$-decomposable. By Proposition~\ref{prop:fd7ssdf} the function $h$ is also $\Pi$-decomposable, which completes the proof.
\end{proof}

\begin{example}
%The classes of Boolean functions, nondecreasing Boolean functions, multilinear polynomial functions, nondecreasing multilinear polynomial functions, and lattice polynomial functions are pivotally characterized by the corresponding pivotal functions as described in the previous sections.
\begin{enumerate}
\item[(a)] The class of Boolean functions is $\Pi$-characterized, where $\Pi\colon \{0,1\}^3\to \{0,1\}$ is defined by $\Pi(p,u,v)=(p\wedge u)\vee(\overline{p}\wedge v)$.

\item[(b)] The class of nondecreasing Boolean functions is $\Pi$-characterized, where $\Pi\colon \{0,1\}^3\to \{0,1\}$ is defined by $\Pi(p,u,v)=(p\wedge (u\vee v))\vee(\overline{p}\wedge (u\wedge v))$.

\item[(c)] The class of multilinear polynomial functions is $\Pi$-characterized, where $\Pi\colon [0,1]^3\to\R$ is defined by $\Pi(p,u,v)=p\, u+(1-p)\, v$.

\item[(d)] The class of nondecreasing multilinear polynomial functions is $\Pi$-characterized, where $\Pi\colon [0,1]^3\to\R$ is defined by $\Pi(p,u,v)=p\, (u\vee v)+(1-p)(u\wedge v)$.

\item[(e)] The class of lattice polynomial functions on a bounded distributive lattice $X$ is $\Pi$-characterized, where $\Pi\colon X^3\to X$ is defined by $\Pi(p,u,v)=\med(p,u,v)$.
\end{enumerate}
\end{example}

\begin{example}
The subclass of $U=\bigcup_{n\geqslant 1}\R^{[0,1]^n}$ of functions that are equivalent to a function $g_{c,n}\colon [0,1]^n\to\R : \bfx\mapsto 1+c\,\prod_{i=1}^nx_i$, where $c\in\R$ and $n\in\N$, is a subclass of the class of multilinear polynomial functions which is $\Pi$-characterized, where $\Pi\colon D\to\R$ is the function $\Pi(p,u,v)=p\, u+(1-p)\, v$ defined on $D=[0,1]\times\R\times\{1\}$. Equivalently, we can consider $\Pi'(p,u,v)=p\, u+(1-p)$ on $D'=[0,1]\times\R^2$.
\end{example}

%---------------------------------------------------------------------------------------------- Section 3
\section{Classes characterized by their unary members}

Proposition \ref{prop:pivotal_unary} shows that a class $\Gamma_{\Pi}$ is characterized by its constant members and the essential unary sections of its members. This observation motivates the following definition, which is inspired from \cite{Cou}.

\begin{definition}\label{de:cu}
We say that a class $C\subseteq U$ is \emph{characterized by its unary members}, or is \emph{UM-characterized}, if it satisfies the following two conditions:
\begin{enumerate}
\item[(i)] A nonconstant function $f$ is in $C$ if and only if so are its essential unary sections.

\item[(ii)] If $f$ is a constant function in $C$ and $g\equiv f$, then $g$ is in $C$.
\end{enumerate}
Equivalently, conditions (i) and (ii) can be replaced by (i) and (ii'), where
\begin{enumerate}
\item[(ii')] If $f$ is in $C$ and $g\equiv f$, then $g$ is in $C$.
\end{enumerate}
We denote by $\umc$ the family of UM-characterized classes $C\subseteq U$.
\end{definition}

\begin{remark}
\begin{enumerate}
\item[(a)] The unary sections considered in condition (i) of Definition~\ref{de:cu} must be essential. Indeed, otherwise for instance the class of multilinear polynomial functions that are strictly increasing in each argument would be considered as a UM-characterized class. However, by adding an inessential argument to any member of this class, the resulting function would no longer be in the class.

\item[(b)] The terminology `unary members' is justified by the fact that the nonconstant unary members of a UM-characterized class $C$ are nothing other than essential unary sections of members of $C$, namely themselves.
\end{enumerate}
\end{remark}

Proposition~\ref{prop:pivotal_unary} shows that every pivotally characterized subclass of $U$ is UM-characterized. As a consequence, a subclass of $U$ that is not UM-characterized cannot be pivotally characterized. Note also that there are UM-characterized subclasses of $U$ that are not pivotally characterized. To give an example, the subclass of all nondecreasing functions in $U$ is UM-characterized but not pivotally characterized (see Example~\ref{ex:LE} for an instance of nondecreasing function with no inessential argument that is not pivotally decomposable).

\begin{example}
The \emph{(discrete) Sugeno integrals} on a bounded distributive lattice $X$ are those lattice polynomial functions on $X$ (see Example~\ref{ex:LP}) which are reflexive (i.e., $f(x,\ldots,x)=x$ for all $x\in X$). Even though the class of lattice polynomial functions is pivotally characterized, the subclass of Sugeno integrals is not UM-characterized and hence cannot be pivotally characterized. Indeed, any unary function $f(x)=x\wedge c$, $c\in X$, is not a Sugeno integral but is an essential unary section of the binary Sugeno integral $g(x_1,x_2)=x_1\wedge x_2$.
\end{example}

The following lemma is an immediate consequence of Definition~\ref{de:cu}.

\begin{lemma}\label{lemma:Csig}
Let $C\subseteq U$ be a UM-characterized class and let $f\colon X^n\to Y$ ($n\geqslant 1$) be a function. Then the following assertions are equivalent.
\begin{enumerate}
\item[(i)] $f\in C$,

\item[(ii)] $f_{\sigma}\in C$ for every permutation $\sigma\colon [n]\to [n]$,

\item[(iii)] every essential section of $f$ is in $C$.
\end{enumerate}
\end{lemma}

We now prove that a subclass of a pivotally characterized class is UM-characterized if and only if it is pivotally characterized (Theorem~\ref{thm:dsu_pivot}). This result will follow from both Proposition~\ref{prop:pivotal_unary} and the following proposition.

For every pivotal function $\Pi$, every $C\subseteq \Gamma_{\Pi}$, every integer $n\geqslant 1$,  and every $k\in [n]$, we set
$$
R_C^{n,k} ~=~ \{(f(\bfx_k^1),f(\bfx_k^0)):\bfx\in X^n, ~f\in C~\mbox{with arity $n$}\}
$$
and we denote by $\Pi_C^{n,k}$ the restriction of $\Pi$ to $X\times R_C^{n,k}$. To simplify the notation we also set $R_C=R_C^{1,1}$ and $\Pi_C=\Pi_C^{1,1}$.

\begin{proposition}
Let $\Pi$ be a pivotal function and consider a UM-characterized subclass $C$ of $\Gamma_{\Pi}$. Then, for every integer $n\geqslant 1$ and every $k\in [n]$, we have $R_C^{n,k}=R_C=\bigcup_{f\in C}R_f$ and $\Pi_C^{n,k}=\Pi_C$. Moreover, $C=\Gamma_{\Pi_C}$.
\end{proposition}

\begin{proof}
Let $n\geqslant 1$ be an integer. We first show that $R_C^{n,k}=R_C^{n,j}$ for all $k,j\in [n]$. Let $(u,v)\in R_C^{n,k}$. Then there exists an $n$-ary function $f\in C$ and an $n$-tuple $\bfa\in X^n$ such that $(u,v)=(f(\bfa_k^1),f(\bfa_k^0))$. Let $\sigma\colon [n]\to [n]$ be the transposition $(jk)$ and let $\bfb$ be the $n$-tuple defined by $b_i=a_j$, if $i=k$, and $b_i=a_i$, otherwise. We then have
$$
(u,v) ~=~ (f(\bfa_k^1),f(\bfa_k^0)) ~=~ (f(\bfb_j^1\sigma),f(\bfb_j^0\sigma)) ~=~ (f_{\sigma}(\bfb_j^1),f_{\sigma}(\bfb_j^0)).
$$
By Lemma~\ref{lemma:Csig} we have $f_{\sigma}\in C$ and hence $(u,v)\in R_C^{n,j}$. The converse inclusion follows by symmetry and we can therefore set $R_C^n=R_C^{n,1}=\cdots =R_C^{n,n}$.

We now prove that $R_C^n\subseteq R_C^{m}$ for all $n,m\geqslant 1$. Assume first that $n<m$. Any $n$-ary function $f\in C$ is equivalent to an $m$-ary function $g$ obtained from $f$ by adding $m-n$ inessential arguments.\footnote{Formally, it suffices to set $g=f_{\iota}$, where $\iota\colon [n]\to [m] : k\mapsto k$. Then $f=g_{\sigma}$, where $\sigma\colon [m]\to [n]$ is an extension of $\iota$ to $[m]$.}
%\footnote{Formally, it suffices to set $g=f_{\sigma}$ for an extension $\sigma\colon [m]\to [n]$ of the identity map $\iota\colon [n]\to [n]$.}
Thus $g\in C$ and, therefore, $R_C^n\subseteq R_C^{m}$. Assume now that $n>m$. The constant functions in $R_C^n$ are also in $R_C^m$ by condition (ii) of Definition~\ref{de:cu}. For every $\bfa\in X^{n-m}$, let $E_{\bfa}$ be the set of functions $g\colon X^m\to Y$ such that there exists a nonconstant $n$-ary function $f\in C$ such that
$$
g(x_1,\ldots,x_m) ~=~ f(a_1,\ldots,a_{n-m},x_1,\ldots,x_m)
$$
for every $\bfx\in X^m$ (up to equivalence, we may assume that the $n$-th argument of $f$ is essential). It follows that
\begin{equation}\label{eq:EcnRb}
R_C^n ~=~ R_C^{n,n} ~=~ \bigcup_{\bfa\in X^{n-m}}\{(g(\bfx_m^1),g(\bfx_m^0)):\bfx\in X^m,~g\in E_{\bfa}\}.
\end{equation}
Since every $g\in E_{\bfa}$ is an $m$-ary essential section of $f$, by Lemma~\ref{lemma:Csig} we have that $g\in C$. Therefore Eq.~(\ref{eq:EcnRb}) means that $R_C^n\subseteq R_C^{m}$.

Thus, we have proved that $R_C^{n,k}=R_C^n=R_C$, and hence $\Pi_C^{n,k}=\Pi_C$ for every integers $n\geqslant 1$ and $k\in [n]$. Let us now prove that $C=\Gamma_{\Pi_C}$.

Since $C\subseteq \Gamma_{\Pi}$, every nonconstant (resp.\ constant) function $f\in C$ is equivalent to a $\Pi$-decomposable function $g$ with no inessential (resp.\ no essential) argument. By condition (ii') of Definition~\ref{de:cu}, $g$ is $\Pi_C$-decomposable. Therefore, $C\subseteq \Gamma_{\Pi_C}$.

To show the converse inclusion, take $h\in \Gamma_{\Pi_C}$ of arity $n$ and let $g\colon X^m\to Y$ be a $\Pi_C$-decomposable function equivalent to $h$ with no inessential argument or no essential argument. If $k\in [m]$ and $\bfa\in X^m$, then $(g(\bfa_k^1),g(\bfa_k^0))\in R_C=R_C^{1,1}$. Thus, there exists a unary function $f\in C$ such that $(g(\bfa_k^1),g(\bfa_k^0))=(f(1),f(0))$. Hence
\begin{equation}\label{eq:s6d7ff}
g(\bfa_k^x) ~=~ \Pi_C(x,f(1),f(0))
\end{equation}
for every $x\in X$.

We have the following two exclusive cases to consider:
\begin{itemize}
\item Suppose that $f$ is a constant function. Since $f\in C\subseteq \Gamma_{\Pi}$, this function is equivalent to a $\Pi$-decomposable constant function $c$. We then have $c=\Pi(x,c,c)$ for every $x\in X$. Since $(c,c)\in R_C$, Eq.~(\ref{eq:s6d7ff}) reduces to $g(\bfa_k^x) = \Pi_C(x,c,c)=c$ for every $x\in X$. Therefore, the constant section $g(\bfa_k^x)$ is equivalent to a function $c$ in $C$.

\item Suppose that $f$ is a nonconstant function. Then $f$ is its own essential unary section. Since $f$ is in $C$, it is $\Pi_C$-decomposable. Therefore, the function defined by the right-hand side of Eq.~(\ref{eq:s6d7ff}) is exactly $f$ and is in $C$.
\end{itemize}
Thus, we have proved that $h$ is equivalent to a function $g$ whose every unary section is in $C$. Hence $g$, and so $h$, are in $C$.
\end{proof}

\begin{theorem}\label{thm:dsu_pivot}
Let $\Pi$ be a pivotal function. A nonempty subclass $C$ of $\Gamma_\Pi$ is UM-characterized if and only if it is pivotally characterized. Moreover, if any of these conditions holds, then $C=\Gamma_{\Pi_C}$.
\end{theorem}

The following corollary immediately follows from Theorem~\ref{thm:dsu_pivot}.

\begin{corollary}\label{cor:dsu_pivot}
If $\Gamma_{\Pi'}\subseteq \Gamma_{\Pi}$ for pivotal functions $\Pi\colon D\to Y$ and $\Pi'\colon D'\to Y$, then $\Pi'=\Pi|_{D''}$, where $D''=X\times R_{\Gamma_{\Pi'}}$.
\end{corollary}

It is sometimes possible to provide additional information about the pivotal function that characterizes a pivotally characterized subclass of a given pivotally characterized class. The next proposition and its corollary illustrate this observation.

\begin{proposition}\label{prop:q5w76e5}
Let $\Pi$ be a pivotal function and let $C$ be a pivotally characterized subclass of $\Gamma_{\Pi}$. Suppose that there exist functions $g,h\colon Y^2\to Y$ such that
\begin{enumerate}
\item[(i)] $(g(u,v),h(u,v))\in R_{\Gamma_{\Pi}}$ for all $(u,v)\in Y^2$,

\item[(ii)] $(g(u,v),h(u,v))=(u,v)$ if and only if $(u,v)\in R_C$.
\end{enumerate}
Then we have $C=\Gamma_{\Pi'}$, where $\Pi'\colon X\times Y^2\to Y$ is defined by $\Pi'(p,u,v)=\Pi(p,g(u,v),h(u,v))$.
\end{proposition}

\begin{proof}
Let us prove that $C\subseteq \Gamma_{\Pi'}$. Let $e\in C\subseteq \Gamma_{\Pi}$ and let $f\colon X^n\to Y$ be a $\Pi$-decomposable function with no essential argument or no inessential argument and equivalent to $e$. By condition (ii) we have
$$
f(\bfx) ~=~ \Pi(x_k,f(\bfx_k^1),f(\bfx_k^0)) ~=~ \Pi'(x_k,f(\bfx_k^1),f(\bfx_k^0)){\,},\qquad\bfx\in X^n,~k\in [n],
$$
which shows that $f$ is $\Pi'$-decomposable and hence that $e\in \Gamma_{\Pi'}$.

To see that $\Gamma_{\Pi'}\subseteq C$, take $e\in \Gamma_{\Pi'}$ and let $f\colon X^n\to Y$ be a $\Pi'$-decomposable function with no essential argument or no inessential argument and equivalent to $e$. We then have
\begin{equation}\label{eq:fs5f7f}
f(\bfx) ~=~ \Pi\big(p,g(f(\bfx_k^1),f(\bfx_k^0)),h(f(\bfx_k^1),f(\bfx_k^0))\big){\,},\qquad\bfx\in X^n,~k\in [n].
\end{equation}
Combining condition (i) and Fact~\ref{fact:FcEq}, we see that $f(\bfx_k^1)=g(f(\bfx_k^1),f(\bfx_k^0))$ and $f(\bfx_k^0)=h(f(\bfx_k^1),f(\bfx_k^0))$ for all $\bfx\in X^n$ and $k\in [n]$. It follows that Eq.~(\ref{eq:fs5f7f}) reduces to the condition that $f$ is $\Pi$-decomposable. Moreover, by condition (ii) we have $(f(\bfx_k^1),f(\bfx_k^0))\in R_C$ for all $\bfx\in X^n$ and $k\in [n]$. Therefore, combining Fact~\ref{rem:bij} and condition (i) of Definition~\ref{de:cu}, we have that $f$ and hence $e$ are in $C$.
\end{proof}

\begin{corollary}\label{cor:q5w76e5}
Assume that $U=\bigcup_{n\geqslant 1}\R^{[0,1]^n}$. Let $\Pi\colon [0,1]\times\R^2\to\R$ be a pivotal function such that $R_{\Gamma_{\Pi}}=\R^2$ and let $C$ be the class of functions $f$ of $\Gamma_{\Pi}$ such that $f(\bfx_k^0)\leqslant f(\bfx_k^1)$ for all $\bfx\in [0,1]^n$ and all integers $n\geqslant 1$ and $k\in [n]$. Then we have $C=\Gamma_{\Pi'}$, where $\Pi'(p,u,v)=\Pi(p,u\vee v,u\wedge v)$ on $[0,1]\times\R^2$.
\end{corollary}

\begin{proof}
By Theorem~\ref{thm:dsu_pivot}, $C$ is pivotally characterized. The result then follows from Proposition~\ref{prop:q5w76e5}.
\end{proof}

Theorem~\ref{thm:dsu_pivot} is also useful to show that the family $\umc$ (see Definition~\ref{de:cu}) can be endowed with a structure of a complete and atomic Boolean algebra.

For any subclass $V\subseteq U$, we denote by $C_V$ the class of the functions whose essential unary sections are in $V$ or that are equivalent to a constant function in $V$.

\begin{theorem}
Let $\mathcal{UMC}=\langle\umc,\vee,\wedge,0,1\rangle$ be the algebra of type $(2,2,0,0)$ defined by $C\wedge D=C\cap D$, $C\vee D=\bigcap\{E\in \umc : E\supseteq C\cup D\}$, $0=\emptyset$ and $1=U$.  The algebra $\mathcal{UMC}$ is a complemented distributive lattice, hence, a Boolean algebra. Moreover, it is complete and atomic. Furthermore, for any pivotal function $\Pi$, the set of subclasses of $\Gamma_{\Pi}$ that are empty or pivotally characterized is equal to the downset generated by $\Gamma_{\Pi}$ in $\mathcal{UMC}$.
\end{theorem}

\begin{proof}
Let us order $\umc$ by inclusion. Clearly, every family $\{A_i : i \in I\}$ of elements of $\umc$ has an \emph{infimum} given by $\bigcap\{A_i : i \in I\}$. Moreover, since $U$ is an element of $\umc$, the class $\bigcap\{E \in \umc : E\supseteq \bigcup_{i \in I}A_i\}$ is the \emph{supremum} $\bigvee_{i \in I} A_i$ of the family $\{A_i : i \in I\}$. Note that $\bigvee_{i \in I} A_i$ contains $f$ if and only if either $f$ is a nonconstant function whose essential unary sections are in $\bigcup_{i \in I} A_i$ or $f$ is a constant function that is in $\bigcup_{i \in I} A_i$.

Distributivity follows directly from the definitions.

For any $A\in \umc$ we denote by $A^*$ the set of the functions that are either constant and not equivalent to an element of $A$ or whose every essential unary section is in $U \setminus A$. Then (i) the essential unary sections of elements of $A^*$ are in $U\setminus A$ and (ii) any nonconstant unary function of $U\setminus A$ is in $A^*$. It follows that $A^*$ is in $\umc$. Indeed, since the case of constant functions is trivial, it suffices to prove that a nonconstant function $f$ is in $A^*$ if and only if every of its essential unary sections  is in $A^*$. First assume that $f$ is in $A^*$. By (i) and (ii), its essential unary sections are in $A^*$. Conversely, if any essential unary section of $f$ is in $A^*$, then by (i) and the definition of $A^*$ we see that $f$ is in $A^*$.

Clearly, $A \wedge A^*=\varnothing$. By construction, we also have $A \vee A^*=U$.

Moreover, $\mathcal{UMC}$ is easily seen to be atomic if we note that its atoms are exactly the classes $d/{\equiv}$ (where $d$ is a constant function) and $C_{\{f\}}$ (where $f$ is a nonconstant unary function).

The last statement is a direct consequence of Theorem~\ref{thm:dsu_pivot}.
\end{proof}

\begin{corollary}\label{cor:cus_iso}
The map $\psi\colon 2^{Y^X}\to\umc :V\mapsto C_V$ is an isomorphism of Boolean algebras.
\end{corollary}

Applying Corollary~\ref{cor:cus_iso} to the special case where $X=Y=\{0,1\}$, we obtain that there are exactly 16 UM-characterized classes of Boolean functions. Each of these classes is of the form $C_V$ for a set of unary Boolean functions $V$. We provide the description of each of these 16 classes in Appendix~\ref{sec:app}.

%---------------------------------------------------------------------------------------------- Section 4
\section{Componentwise pivotal decompositions}

In this section we generalize the concept of pivotal decomposition by allowing the pivotal functions to depend upon the label of the pivot variable. Let $X_1,\ldots,X_n$ and $Y$ be nonempty sets and, for every $k\in [n]$, let $0_k$ and $1_k$ be two fixed elements of $X_k$. When no confusion arises we simply denote $0_k$ and $1_k$ by $0$ and $1$, respectively. For every function $f\colon \prod_{i=1}^n X_i\to Y$, define $R_f^k=\{(f(\bfx_k^1),f(\bfx_k^0)):\bfx\in \prod_{i=1}^n X_i\}$.

\begin{definition}\label{de:CompPD}
We say that a function $f\colon \prod_{i=1}^n X_i\to Y$ \emph{admits a componentwise pivotal decomposition}, or is \emph{c-pivotally decomposable}, if there exist subsets $D_k$ of $X_k\times Y^2$ and functions $\Pi_k\colon D_k\to Y$, $k=1,\ldots,n$, called \emph{pivotal functions}, such that, for every $k\in [n]$, we have $D_k\supseteq X_k\times R_f^k$ and
\begin{equation}\label{eq:65q7eq67}
f(\bfx) ~=~ \Pi_k(x_k,f(\bfx_k^1),f(\bfx_k^0)){\,},\qquad \bfx\in \prod_{i=1}^n X_i.
\end{equation}
In this case we say that $f$ is $(\Pi_1,\ldots,\Pi_n)$-decomposable.
\end{definition}

Clearly, Facts~\ref{fact:FcEq} and \ref{fact:Restr01} and Proposition~\ref{prop:Uniq} can be easily extended to the case of c-pivotally decomposable functions. We also have the following fact, which is the counterpart of Fact~\ref{rem:bij}.

\begin{fact}\label{fact:NotCPD}
Eq.~(\ref{eq:65q7eq67}) exactly means that, for every fixed $\bfa,\bfb\in \prod_{i=1}^n X_i$ and $k\in [n]$, we have $f_k^{\bfa}=f_k^{\bfb}$ if and only if $(f(\bfa_k^1),f(\bfa_k^0))=(f(\bfb_k^1),f(\bfb_k^0))$.
\end{fact}

A function that is pivotally decomposable is clearly c-pivotally decomposable. The following example shows that there are c-pivotally decomposable functions that are not pivotally decomposable. There are also functions that are not c-pivotally decomposable.

\begin{example}\label{ex:LE}
The \emph{Lov\'asz extension} of a pseudo-Boolean function $f\colon \{0,1\}^n\to\R$ is the unique function $L_f\colon [0,1]^n\to\R$ of the form
$$
L_f(\bfx) ~=~ \sum_{S\subseteq [n]}a_S\,\bigwedge_{i\in S}x_i{\,},\qquad a_S\in\R{\,},
$$
%$$
%L_f(\bfx) ~=~ \sum_{S\subseteq [n]}f(\mathbf{1}_S)\,\sum_{T\supseteq S}(-1)^{|T|-|S|}\,\bigwedge_{i\in T}x_i{\,}.
%$$
that agrees with $f$ on $\{0,1\}^n$ (see, e.g., \cite{GraMarRou00} and the references therein). We then have
$$
f(\mathbf{1}_T)=\sum_{S\subseteq T}a_S\quad\mbox{and}\quad a_S=\sum_{T\subseteq S}(-1)^{|S|-|T|}\, f(\mathbf{1}_T).
$$

%The so-called \emph{discrete Choquet integrals} are those Lov\'asz extensions which are nondecreasing and reflexive.

Every binary Lov\'asz extension $L_f\colon [0,1]^2\to\R$ is c-pivotally decomposable. %However, in general such a function is not pivotally decomposable.
Indeed, consider the binary Lov\'asz extension
$$
L_f(x_1,x_2) ~=~ a_0+a_1{\,}x_1+a_2{\,}x_2+a_{12}{\,}(x_1\wedge x_2)
$$
and construct $\Pi_1\colon [0,1]\times\R^2\to\R$ as follows (we construct $\Pi_2$ similarly). If $a_2\neq 0$, then
$$
\Pi_1(p,u,v) ~=~ a_0+a_1{\,}p+(v-a_0)+a_{12}{\,}\big(p\wedge\frac{v-a_0}{a_2}\big).
$$
If $a_2=0$ and $a_{12}\neq 0$, then
$$
\Pi_1(p,u,v) ~=~ a_0+a_1{\,}p+a_{12}{\,}\big(p\wedge\frac{u-a_0-a_1}{a_{12}}\big).
$$
If $a_2=0$ and $a_{12}=0$, then $\Pi_1(p,u,v)=a_0+a_1{\,}p$.

There are ternary Lov\'asz extensions $L_f\colon [0,1]^3\to\R$ that are not c-pivotally decomposable. Indeed, considering for instance $L_f(x_1,x_2,x_3) = x_1\wedge x_2+x_2\wedge x_3$ with $\bfa=(1/2,1/2,1/2)$ and $\bfb=(1/4,1/2,3/4)$, we have $a_2=1/2=b_2$, $L_f(\bfa_2^1)=1=L_f(\bfb_2^1)$, $L_f(\bfa_2^0)=0=L_f(\bfb_2^0)$, and $L_f(\bfa)=1\neq 3/4=L_f(\bfb)$. By Fact~\ref{fact:NotCPD}, this shows that $L_f$ is not c-pivotally decomposable.
\end{example}

The following two examples provide classes of functions that are c-pivotally decomposable but not necessarily pivotally decomposable.

\begin{example}\label{ex:psLP}
Let $X_1,\ldots,X_n$ and $Y$ be bounded distributive lattices, with $0$ and $1$ as bottom and top elements, respectively. A function $f\colon \prod_{i=1}^n X_i\to Y$ is of the form $f=g\circ (\phi_1,\ldots,\phi_n)$, where $g\colon Y^n\to Y$ is a lattice polynomial function and the $\phi_i\colon X_i\to Y$, $i=1,\ldots,n$, are unary functions such that $\phi_i(x)=\med(\phi_i(x),\phi_i(1),\phi_i(0))$ for every $x\in X_i$, if and only if it is $(\Pi_1,\ldots,\Pi_n)$-decomposable, where $\Pi_k\colon X_k\times Y^2\to Y$ is defined by $\Pi_k(p,u,v)=\med(\phi_k(p),u,v)$; see \cite{Cou2011}.
\end{example}

\begin{example}\label{ex:mBf}
A pseudo-Boolean function $f\colon \{0,1\}^n\to \R$ is \emph{monotone} if it is either isotone or antitone in each of its arguments. It can be easily seen \cite[Theorem~1]{CouMarWal12} that a pseudo-Boolean function is monotone if and only if it is of the form $f=g\circ (\phi_1,\ldots,\phi_n)$, where $g\colon [0,1]^n\to \R$ is a nondecreasing pseudo-Boolean function and each $\phi_k\colon \{0,1\}\to \{0,1\}$ is either the identity function $\phi_k=\id$ or the negation function $\phi_k=\neg$. Applying Example~\ref{ex:psLP} to the special case where $X_1=\cdots =X_n=\{0,1\}$ and $Y=\R$, we see that a pseudo-Boolean function is monotone if and only if it is $(\Pi_1,\ldots,\Pi_n)$-decomposable, where $\Pi_k\colon \{0,1\}^3\to\R$ is defined by $\Pi_k(p,u,v)=\med(\phi_k(p),u,v)$.
\end{example}

\section{Conclusions and further research}

In this paper we have introduced and investigated a generalization of the Shannon decomposition called pivotal decomposition. Considering the wide number of applications of the Shannon decomposition for Boolean functions, the concept of pivotal decomposition can prove to be a useful tool to study structural properties of classes of functions arising from various areas such as fuzzy system theory, fuzzy game theory, and aggregation function theory. We list a few ideas of possible applications or further investigations.
\begin{enumerate}
\item[(a)] Repeated applications of the Shannon decomposition lead to \emph{median normal forms} of monotone Boolean functions. It is known \cite{Couceiro2006,Cou11} that median normal form systems are of lower complexity than the disjunctive and conjunctive normal form systems. Similarly, the existence of a pivotal decomposition for a class of functions also leads to normal form representations. Comparing the complexity of these representations and designing  efficient algorithms to obtain them are two important problems that could be addressed to foster applications of pivotal decomposition.

\item[(b)] It is known \cite{Bioch2010} that the Shannon decomposition can be used as a tool to analyze the decomposability of a Boolean function. Recall that a Boolean function $f\colon\{0,1\}^n \to \{0,1\}$ is \emph{decomposable} \cite{Bioch2010} if there exists a partition $\{A_1, \ldots, A_\ell\}$ of $[n]$ and functions $F\colon \{0,1\}^\ell \to\{0,1\}$ and $g_i\colon \{0,1\}^{|A_i|} \to \{0,1\}$ for $i=1,\ldots,\ell$ such that
\[
f(\bfx)~=~ F(g(x_i)_{i\in A_1}, \ldots, g(x_i)_{i\in A_\ell}){\,}, \qquad \bfx\in \{0,1\}^n.
\]
Decomposability of Boolean functions corresponds to interpretable properties in applied areas such as game theory \cite{Shapley1967} and system reliability \cite{Birnbaum1965}. In various contexts, such as aggregation function theory, generalized versions of the decomposability property can reveal interesting structural properties of certain classes of functions. The existence of a pivotal decomposition could then ease the analysis of decomposability.

\item[(c)] It would be interesting to generalize Definition~\ref{de:345} by considering two pivots instead of one. Then the quest for functions that are pivotally decomposable with two pivots and not pivotally decomposable with one pivot could be an interesting question.

\item[(d)] Algebraic properties of UM-characterized classes of functions and connections with clone theory could also be investigated.
\end{enumerate}

%----------------------------------------------------------------------------------------------
\appendix
%---------------------------------------------------------------------------------------------- Appendix
\section{UM-characterized classes of Boolean functions}\label{sec:app}

We use the following notation. For any $\bfa\in \{0,1\}^n$ we denote by $\chi_{\bfa}$ the characteristic function of $\bfa$, i.e., the Boolean function defined on $\{0,1\}^n$ by $\chi_\bfa(\bfx)=1$ if and only if $\bfx=\bfa$. In fact, $\chi_\bfa(\bfx)=\prod_{\{i:a_i=1\}}x_i$. The bottom element of $\{0,1\}^n$ is denoted by $\bold{0}_n$ or by $\bold{0}$ if no confusion arises and the top by $\bold{1}_n$ or by $\bold{1}$. We denote by $\bold{B}$ the class of the Boolean functions.

For any $f\colon\{0,1\}^n\to\{0,1\}$ and any $j\in [n]$, we denote by $\partial_j f$ and $\Delta_j f$ the \emph{$j$-th partial derivatives of $f$}, i.e., the functions defined by
\begin{eqnarray*}
&& \partial_j f:\{0,1\}^n \to \{0,1\}: \bold{x}\mapsto f(\bold{x}\oplus \bold{\delta}_j)\oplus f(\bold{x}){\,},\\
&& \Delta_j f:\{0,1\}^n \to \{-1,0,1\}: \bold{x}\mapsto f(\bold{x}_j^1)-f(\bold{x}_j^0){\,},
\end{eqnarray*}
where the map $\delta_j\in\{0,1\}^{[n]}$ is defined by $\bold{\delta}_j(k)=1$ if and only if $k=j$.

\begin{proposition}
Let us denote by $\ctop, \cbot, \id$ and $\neg$ the 4 unary Boolean functions defined according to their truth tables:
\[
\begin{array}{r|l|l|l|l}
 & \ctop & \cbot & \id & \neg\\ \hline
0 & 1 & 0 & 0 & 1\\
1 & 1 & 0 & 1 & 0.
\end{array}
\]
The 16 UM-characterized classes of Boolean functions can be described as follows:
\begin{enumerate}
%\begin{multicols}{2}
\item[(1)] $C_{\{\varnothing\}} ~=~ \varnothing$
\item[(2)] $C_{\{\cbot, \ctop, \id, \neg\}} ~=~ \bold{B}$
\item[(3)] $C_{\{\cbot\}} ~=~ \cbot/{\equiv}$
\item[(4)] $C_{\{\ctop\}} ~=~ \ctop/{\equiv}$
\item[(5)] $C_{\{\id\}} ~=~ \id/{\equiv}$
\item[(6)] $C_{\{\neg\}} ~=~ \neg/{\equiv}$
\item[(7)] $C_{\{\cbot, \ctop\}} ~=~ \{\cbot,\ctop\}/{\equiv}$
\item[(8)] $C_{\{\cbot, \id\}} ~=~ \bigcup\{\chi_{\ctop_n}/{\equiv}: n\geqslant 1\}\cup\cbot/{\equiv}$
\item[(9)] $C_{\{\cbot, \neg\}} ~=~ \bigcup\{\chi_{\bold{0}_n}/{\equiv}: n\geqslant 1\}\cup\cbot/{\equiv}$
\item[(10)] $C_{\{\ctop, \id\}} ~=~ \bigcup\{\chi_{\bold{0}_n}/{\equiv}: n\geqslant 1\} \cup\ctop/{\equiv}$
\item[(11)] $C_{\{\ctop, \neg\}} ~=~ \bigcup\{\chi_{\bold{1}_n}/{\equiv}: n\geqslant 1\} \cup\ctop/{\equiv}$
\item[(12)] $C_{\{\id, \neg\}} ~=~ \{f : \forall j ~ (\partial_j f = \ctop ~\vee ~\partial_j f = \cbot)  \}$
\item[(13)] $C_{\{\cbot, \id, \neg\}} ~=~ \{f : \forall j ~ (\partial_j f \geqslant f ~\vee ~\partial_j f = \cbot)  \}$
\item[(14)] $C_{\{\ctop, \id, \neg\}} ~=~ \{f : \forall j ~ (\partial_j f \leqslant f ~\vee ~\partial_j f = \cbot)\}$
\item[(15)] $C_{\{0, 1, \id\}} ~=~ \{f : \forall j ~ \Delta_j f\geqslant 0\}$
\item[(16)] $C_{\{0, 1, \neg\}} ~=~ \{f : \forall j ~ \Delta_j f\leqslant 0\}$
%\end{multicols}
\end{enumerate}
\end{proposition}

\begin{proof}
(1), (2), (3), and (4) are trivial.

(5) We have to prove that $C_{\{\id\}}\subseteq \id/{\equiv}$. First note that $C_{\{\id\}}$ does not contain any constant function (such a function would be equivalent to a unary constant function of $C_{\id}$ which does not contain any unary constant function). Then, let $f\colon\{0,1\}^n \to\{0,1\}$ be an element of $C_{\id}$ and assume that the $k$-th argument of $f$ is essential. It follows that $f_k^{\bfa}=\id$ for every $\bfa \in \{0,1\}^n$. If $j\neq k$, it follows that, for every $\bfa \in \{0,1\}^n$ we have
\[
f_{j}^{\bfa_k^0}=\cbot \quad \mbox{and} \quad f_{j}^{\bfa_k^1}=\ctop{\,},
\]
which means that the $j$-th argument of $f$ is inessential.  Hence the function $f$ is equivalent to the identity function.

(6) is obtained similarly as in (5).

(7) We have to prove that $C_{\{\cbot, \ctop\}}\subseteq \{\cbot,\ctop\}/{\equiv}$. Since $\{\cbot,\ctop\} \subseteq \{\cbot, \ctop\}/{\equiv}$ it suffices to prove that $C_{\{\cbot, \ctop\}}$ does not contain any nonconstant function. Assume that $f\colon\{0,1\}^n \to\{0,1\}$ is an element of $C_{\{\cbot,\ctop\}}$ whose $k$-th argument is essential. Then, for every $\bfa\in\{0,1\}^n$, the section $f_k^{\bfa}$ is in $\{\cbot, \ctop\}$ and hence is a constant function, a contradiction.

(8) By definition $\cbot\in C_{\{\cbot,\id\}}$. The function $f=\chi_{\ctop_n}$ is in $C_{\{\cbot,\id\}}$ for every $n \geqslant 1$ since for every $k\in [n]$ and every $\bfa \in \{0,1\}^n$ the unary section $f_k^\bfa$ is the zero function if there is a $j \neq k$ such that $a_j=0$ and $f_k^\bfa$ is the identity function otherwise. From the fact that $C_{\{\cbot,\id\}}$ is $\equiv$-saturated, we deduce that $\bigcup\{\chi_{\ctop_n}/{\equiv}: n\geqslant 1\}\cup \cbot/{\equiv}\subseteq C_{\{\cbot,\id\}}$.

Let us prove the converse inclusion. We prove that if the $k$-th and $j$-th arguments ($j \neq k$) of an element $f\colon\{0,1\}^n \to\{0,1\}$ of $C_{\{\cbot,\id\}}$ are essential then $f_k^{\bfa}=0$ for every $\bfa\in \{0,1\}^n$ such that $a_j=0$. Indeed, if $f_k^{\bfa}=\id$ then $f(\bfa^1_k)=f(\bfa^{10}_{kj})=1$. It follows that if $\bold{b}=\bfa_k^1$, then $f_j^{b} \in \{\ctop, \neg\}$ and $f$ cannot be in $C_{\{\cbot, \id\}}$.

Hence $f(\bfa)$ vanishes as soon as there is an essential argument of $f$ that is set to $0$. Then, if $f(\ctop)=0$, the function $f$ is in $\cbot/{\equiv}$, and if $f(\ctop)=1$, it is in $\chi_{\ctop_n}/{\equiv}$.

(9) We proceed similarly as in (8). In this case, if $f$ is in $C_{\{\cbot,\neg\}}$ and if the $k$-th and $j$-th arguments of $f$ (with $k \neq j$) are essential then $f_k^\bfa=0$ if $a_j=1$.

(10) is obtained from (8) by duality.

(11) is obtained from (9) by duality.

For (12), (13), and (14) we first note that the $j$-th argument of $f\colon\{0,1\}^n\to\{0,1\}$ is inessential if and only if $\partial_j f=\cbot$.

(12) is easy if we note that $f$ is in $C_{\{\id,\neg\}}$ if and only if, for every essential argument $j$ of $f$, we have $\partial_j f=\ctop$.

(13) Assume that $f\colon\{0,1\}^n \rightarrow \{0,1\}$ is in $\{f : \forall j ~(\partial_j f \geqslant f \vee \partial_j f=\cbot)\}$. If $j\in [n]$ is such that $\partial_j f \geqslant f$ then for every $\bfa \in \{0,1\}^n$ it follows that if $f(\bfa_j^0)=1$ then $f(\bfa_j^1)=0$ and if $f(\bfa_j^1)=1$ then $f(\bfa_j^0)=0$. Hence, any essential unary section of $f$ can be any unary Boolean function but $\ctop$.

Conversely, assume that $f\colon\{0,1\}^n \to\{0,1\}$ is in $C_{\{\cbot, \id, \neg\}}$. If the $k$-th argument ($k\in [n]$) of $f$ is essential then for every $\bfa \in \{0,1\}^n$ the function $f_k^{\bfa}$ cannot be equal to $\ctop$. It means that if $f(\bfa_k^0)=1$ then $f(\bfa_k^1)=0$ and if $f(\bfa_k^1)=1$ then $f(\bfa_k^0)=0$, which proves that $\partial_j f \geqslant f$.

(14) is obtained from (13) by duality.

(15) and (16) are examples that have already been considered (these are the set of the nondecreasing functions and the set of the nonincreasing functions, respectively).
\end{proof}

\end{document}